\documentclass[11pt,leqno]{amsart}

\usepackage{graphicx}
\usepackage{epstopdf,epsfig}
\usepackage[hidelinks]{hyperref}

\usepackage{xcolor}

\usepackage{amsmath,amsthm,amsfonts,amssymb,latexsym,amscd,enumerate}

\usepackage{palatino}
\usepackage[mathcal]{euler}

\usepackage{xy}
\xyoption{all}

\swapnumbers

\theoremstyle{plain}
\newtheorem*{theoremA*}{Theorem A}
\newtheorem*{theoremB*}{Theorem B}

\newtheorem*{theorem*}{Theorem} 
\newtheorem*{lemma*}{Lemma}
\newtheorem*{assumption*}{Assumption}

\newtheorem{theorem}[equation]{Theorem} 
\newtheorem{lemma}[equation]{Lemma}

\theoremstyle{definition}
\newtheorem{definition}[equation]{Definition}
\newtheorem{example}[equation]{Example}

\newtheorem{remark}[equation]{Remark}

\newtheorem*{observation*}{Observation}

\theoremstyle{remark}

\numberwithin{equation}{subsection}

\numberwithin{equation}{section}

\newcommand{\R}{\mathbb{R}}
\newcommand{\C}{\mathbb{C}} 

\newcommand{\Z}{\mathbb{Z}}

\newcommand{\A}{{A}}

\newcommand{\Calg}{C^\ast\text{-algebra}}

\newcommand{\D}{\mathcal{D}}
\newcommand{\CAT}{\operatorname{CAT}}
\newcommand{\Mid}{\operatorname{Mid}}
\newcommand{\SAH}{\operatorname{SAH}}

\newcommand{\Compacts}{\mathfrak{K}}

\begin{document} 

\title[Baum--Connes Conjecture for Groups Acting on CAT(0)-Cubical Spaces]{On the Baum--Connes Conjecture for Groups Acting on CAT(0)-Cubical Spaces} 
\author{Jacek Brodzki}
\author{Erik Guentner}
\author{Nigel Higson}
\author{Shintaro Nishikawa}

\address{J.B.: Building 54, Mathematical Sciences, University of Southampton, Highfield, Southampton SO17 1BJ, England.}
\address{E.G.: Department of Mathematical Sciences,
University of Hawaii at Manoa,
2565 McCarthy Mall, Keller 401A,
Honolulu, HI 96822, USA.}
\address{N.H \& S.N.: Department of Mathematics, Penn State University, University Park, PA 16802, USA.} 

\begin{abstract}
We give a new proof of the Baum--Connes conjecture with coefficients
for any second countable, locally compact topological group that acts
properly and cocompactly on a finite-dimensional CAT(0)-cub\-ical space
with bounded geometry.  The proof uses the Julg-Valette complex of a
CAT(0)-cubical space introduced by the first three authors, and the
direct splitting method in Kasparov theory developed by the last
author. 
\end{abstract}

\thanks{J.B.\  was supported in part by EPSRC grant EP/N014189/1.} 
\thanks{E.G.\   was supported in part by   the Simons   Foundation    grant \#586754.}

\maketitle


\section{Introduction}

The computation of $K$-theory groups for crossed product $C^*$-algebras is a central problem in
C$^*$-algebra theory, the solution of which has deep applications to
manifold topology and to the algebra of group
rings.  The {Baum--Connes conjecture} proposes a general form of the solution.  It  asserts
that the \emph{Baum--Connes assembly map}
\begin{equation}
\label{eq-bc-assembly-map}
K_*^G (\underline{E}G; A) \longrightarrow K_*(C^*_r (G,A))
\end{equation}
is an isomorphism for any second countable, locally compact group $G$
and any separable coefficient $G$-$C^*$-algebra $A$. Here  $\underline{E}G$ is a
universal proper $G$-space and $C^*_r (G,A)$ is the reduced crossed
product.  See \cite{BCH94} for details and discussion.  

Our
launching point is the following theorem of Higson and
Kasparov \cite{HigsonKasparov,HK97}:

\begin{theorem*}[Higson, Kasparov]
   Let $G$ be a second countable, locally compact group.  If $G$
   admits an isometric and metrically proper action on a
   \textup{(}possibly infinite dimensional\textup{)}
   Euclidean space then the Baum--Connes assembly map
   \eqref{eq-bc-assembly-map} is an isomorphism for $G$ and any
   separable coefficient algebra $A$.
\end{theorem*}

Gromov introduced the term \emph{a-T-menable} for the groups in 
the statement of this theorem.  
Our purpose here is  to reexamine the Baum--Connes
isomorphism for a class of a-T-menable groups having a very geometric
flavor.  We consider groups which admit a proper, cellular action on a
\emph{bounded geometry} $\CAT(0)$-cubical space.  A $\CAT(0)$-cubical
space is a simply-connected topological space that is  built by gluing
Euclidean cubes along their faces in a manner that satisfies a local combinatorial condition.  A motivating example is a
finite product of simplicial trees.  The bounded geometry condition
asserts that there is a uniform bound on the number of edges to which
a vertex can belong.
%
%

Building  our previous work \cite{BGH19}, we shall prove the following theorem.  While the result itself is a consequence of the Higson-Kasparov theorem, the approach here is direct, geometrical and better rooted in
simple finite-dimensional constructions.

\begin{theoremA*} 
The Baum--Connes assembly map \eqref{eq-bc-assembly-map} is an
isomorphism for a second countable, locally compact topological
group $G$ that acts properly and cocompactly on a bounded geometry $\CAT(0)$-cubical
space and any separable coefficient algebra $A$.
\end{theoremA*}

We shall refer to the groups that we are considering here as
\emph{$\CAT(0)$-cubical groups}.  The essential difference between
$\CAT(0)$-cubical and a-T-menable groups, and so also between the
Higson-Kasparov theorem and our work, is the difference between a
combinatorial and a measured setting.  It is very much analogous to
the difference between a simplicial tree and an $\R$-tree.

A key feature of $\CAT(0)$-cubical spaces is the existence of
\emph{hyperplanes}, and indeed a $\CAT(0)$-cubical space can be essentially
reconstructed from the combinatorics of its hyperplanes.  Each
hyperplane divides the space into precisely two connected components,
and a hyperplane \emph{separates} a pair of points if they are in
different components.  Counting the number of hyperplanes that
separate a pair of points defines a distance and, restricting to a
$G$-orbit, we obtain a distance on $G$ which is intimately related to
the a-T-menability of $G$ through the theory of conditional negative-type functions.  Compare \cite{MR1432323}.
The analogous \emph{measured} notion is that of a space with measured
walls (hyperplanes).  For this notion, the combinatorics of counting
is replaced by a measure on a set of hyperplanes.  The existence of a
proper action on a space with measured walls characterizes
a-T-menability; this was proven for countable discrete groups in
\cite{MR2106770,MR1609459}, and later extended to the locally compact
and second countable setting in \cite{MR2671183}.

The following concrete example is analyzed by Haglund
\cite{haglund-arxiv}.  The Baumslag-Solitar group
\begin{equation*}
  BS(2,3) = \langle\; a,t \;|\; ta^2t^{-1} = a^3 \,\rangle
\end{equation*}
is a-T-menable but cannot act properly on a $\CAT(0)$-cubical space.
Haglund points out that $BS(2,3)$ admits a proper action on the
product of a simplicial tree (its Bass-Serre tree) and the hyperbolic
plane, which is a space with measured walls, and hence $BS(2,3)$ is
a-T-menable.  However, since the cyclic subgroup generated by $a$ is
distorted, that subgroup must have a fixed point in any action of
$BS(2,3)$ on a $\CAT(0)$-cubical space. 

A second, less substantial difference between our result and the Higson-Kasparov theorem is related to the bounded
geometry hypothesis.  Farley proved that Thompson's group $F$ acts
properly on a locally finite but \emph{infinite dimensional}
$\CAT(0)$-cubical space, and so this group is a-T-menable
\cite{MR2006480}.   Nevertheless, $F$ (apparently) falls outside the
scope of our results.  Extending our work to the locally finite but
infinite dimensional setting may be possible, but would entail
significant technical hurdles.

In the balance of this introduction, we shall place this piece in the
broader context of our previous work \cite{BGH19,Nishikawa19} and
explain in outline our new  proof of Theorem~A.  Assume given a second
countable, locally compact group $G$ acting by automorphisms on a bounded
geometry $\CAT(0)$-cubical space.  The Kasparov representation ring of
$G$, denoted $R(G)=KK_G(\C,\C)$ plays a central role here, and indeed
in the Baum--Connes theory in general.  Associated to $G$ there are  three different
Kasparov cycles for $R(G)$.  It emerges that they represent the same
element, but they have different properties, and it is precisely these
differences that allow us to draw conclusions.  The three cycles are:
\begin{eqnarray*}
  D_\mathrm{JV} & : & \text{the Julg-Valette cycle}, \\
  D_\mathrm{PS} & : & \text{the Pytlik-Szwarc cycle}, \\
  D_\mathrm{dR} & : & \text{the de Rham cycle}.
\end{eqnarray*}
The first of these was introduced in \cite{BGH19}, although it will
reappear below in slightly modified form. The second is from \cite{BGH19}, and
the third is the focus of this paper.

A cycle for $R(G)$ is an odd, self-adjoint operator on a
separable, graded Hilbert space $H$ equipped with a representation of
$G$ by grading preserving unitary operators.  The operator satisfies
axioms related to approximate invertibility  and approximate  $G$-equivariance.  The exact form
of the axioms depends on whether one works with bounded or unbounded
cycles, and we refer to Section \ref{sec:direct-splitting} below for precise definitions.

Here then are the differences among the three cycles relevant to this
work.  The Pytlik-Szwarc cycle is constructed directly from the
combinatorics of hyperplanes, and the operator is an \emph{exactly
  $G$-equivariant Fredholm operator of index one.}  It follows
immediately that the Pytlik-Szwarc cycle represents the identity of
$R(G)$.  The unitary representations appearing in the Julg-Valette
cycle are the permutation representations on the set of cubes (of
various dimensions) comprising the $\CAT(0)$-cubical space.  Assuming $G$
acts properly, these representations are \emph{weakly contained in the
  regular representation} of $G$.  The operator in the Julg-Valette
cycle is \emph{not} exactly $G$-equivariant (although, it is
approximately $G$-equivariant, as required by the
axioms).  Finally, assuming the action is both proper and cocompact,
the de Rham cycle satisfies \emph{property 
  ($\gamma$)} recently introduced by Nishikawa \cite{Nishikawa19}. The de Rham cycle
will be defined and analyzed in detail in Sections~4 and 5 below.  We
shall prove Theorem~A by proving the following result:

\begin{theoremB*}
  Let $G$ be a second countable, locally compact topological group
  that acts by automorphisms on a bounded geometry
  $\CAT(0)$-cubical space.  The three cycles above all represent the multiplicative 
  unit of the representation ring:
  \begin{equation}\label{eq-summary}
    [ D_\mathrm{dR} ] = [ D_\mathrm{JV} ] = [ D_\mathrm{PS} ] = 1 \in R(G). 
  \end{equation}
In particular,
\begin{enumerate}[\rm (i)]
\item if $G$ acts properly then it is $K$-amenable, since $[D_\mathrm{JV}]=1$; and
\item if $G$ acts both properly and cocompactly then it satisfies the
  Baum--Connes conjecture, since $[D_\mathrm{dR}]=1$. 
\end{enumerate}
\end{theoremB*}

The equalities (\ref{eq-summary}) are the heart of this
theorem.  Assuming these, and the properties of the various cycles
outlined above, conclusions (i) and (ii) are immediate; the first is
essentially the definition of $K$-amenability \cite{MR716254}, and the
second is an application of Nishikawa's direct splitting method
\cite{Nishikawa19}. 
As for the individual equalities comprising (\ref{eq-summary}), the
first is the main technical result of this piece, and is the content of
Theorem~\ref{thm-de-Rham-is-JV}.  The second is the main technical
result of \cite{BGH19} and, as remarked above, the final equality is
obvious.

The paper is organized as follows.  Sections 2 and 3 contain
background material.  In particular, in Section 2 we review
Nishikawa's direct splitting method from \cite{Nishikawa19} and in
Section 3 we recast the definition of the Julg-Valette cycle from
\cite{BGH19} in a form convenient for what follows.  We introduce the
de Rham cycle in Section 4, and prove that it is an unbounded Kasparov
cycle for $R(G)$.  See Theorem~\ref{thm-de-Rham-cycle}.  Finally, we
analyze the de~Rham cycle in Section~5.  The key results here are the
equality of the de Rham and Julg-Valette cycles in $R(G)$ and that the
de Rham cycle has Property ($\gamma$).  See
Theorems~\ref{thm-de-Rham-is-JV} and \ref{thm-de-Rham-has-gamma},
respectively.

%


\section{The Direct Splitting Method}\label{sec:direct-splitting}

In this section, $G$ will be an arbitrary second countable, locally compact topological group that admits a second countable, locally compact and $G$-compact model $E=\underline{E}G$ of the universal proper $G$-space (see \cite{BCH94} for more information about  $\underline{E}G$).

We begin by recalling the definition of Kasparov cycles for $KK_G(\C,\C)$ from \cite{Ka88}. For a Hilbert space $H$, we denote by $\Compacts(H)$  the algebra of all compact operators on $H$.

\begin{definition} A Kasparov cycle for $KK_G(\C, \C)$ is a pair $(H, F)$ where $H$ is a separable, graded $G$-Hilbert space and $F$ is an odd, self-adjoint bounded $G$-continuous operator on $H$ satisfying the following two conditions:

\begin{enumerate}[\rm (i)]

\item  $1-F^2 \in \Compacts(H)$; and 

\item 
$ g(F)- F\in \Compacts(H) $ for any $g\in G$,
\end{enumerate}
where $g(F) = u_g F  u_g^\ast$, with  $g\mapsto u_g$ the  unitary representation of $G$ on $H$.
\end{definition}
There is a natural notion of homotopy  of Kasparov cycles (see \cite[Definition 2.3]{Ka88}). The commutative ring $R(G)=KK_G(\mathbb{C}, \mathbb{C})$ is defined as the set of homotopy equivalence classes of Kasparov cycles. We write by $[H, F]$, the element in $R(G)$ which corresponds to a cycle $[H, F]$.  
Among the Kasparov cycles there is a  cycle of the form  $[\mathbb{C}\oplus0, 0]$ which corresponds to the trivial representation of $G$. This cycle is denoted by 
 $1_G$ and it is the multiplicative identity element of the ring $R(G)$.

For a locally compact, Hausdorff space $X$ and a $\Calg$ $A$, we denote by $C_0(X, A)$, the $\Calg$ of continuous $A$-valued functions on $X$ vanishing at infinity. Recall that a compactly supported, continuous, non-negative function $c$ on a (co-compact) proper $G$-space $X$ such that the Haar integral 
\[\displaystyle \int_{g \in G}g(c)^2d\mu_G(g)=1
\] is called a cut-off function. We also note that for any closed subgroup $G'$ of $G$, a Kasparov cycle for $KK_G(\C, \C)$ can be viewed as a cycle for $KK_{G'}(\C, \C)$. 

\begin{definition}[See {\cite[Definition 2.2]{Nishikawa19}}]
 \label{def_gamma}
We say a Kasparov cycle $(H, F)$ for $KK_G(\C, \C)$ has Property $(\gamma)$ if the following conditions are satisfied:
\begin{enumerate}[\rm (i)]
\item 
$[H, F]=1_K$ in $R(K)$ for any compact subgroup $K$ of $G$;

\item 
there is a non-degenerate, $G$-equivariant representation of the $G$-$C^*$-alg\-ebra $C_0(E)$ on $H$ satisfying the following:

\begin{enumerate}
\item[(iia)] 
the function $g \mapsto [g(\phi), F ]$ belongs to $C_0(G, \Compacts(H))$ for every $\phi \in C_0(E)$;
and

\item[(iib)] 
for some cut-off function $c$ on $E$
\[
\int_{g \in G} g(c)F g(c)d{\mu_G}(g) - F \in \Compacts(H),
\]
where the Haar integration appearing here is taken in the strong topology.
\end{enumerate}

\end{enumerate}
\end{definition}

The main result of \cite{Nishikawa19} is as follows:

\begin{theorem}[See {\cite[Corollary 5.6]{Nishikawa19}}]
\label{thm_BCC}
 If the element $1_G \in KK_G(\C,\C)$ is represented by a Kasparov cycle $(H,F)$ with Property $(\gamma)$, then the Baum--Connes conjecture with coefficients holds for $G$.
\end{theorem}

We end this section by recalling the definition of unbounded cycles for $KK_G(\C,\C)$ and sufficient conditions for such a cycle to define a (bounded) cycle with Property $(\gamma)$.

\begin{definition}[Compare \cite{BaajJulg}]
\label{def-unbounded-cycle}
An \emph{unbounded cycle} for $KK_G(\C,\C)$ consists of a  separable, graded $G$-Hilbert space $H$ and an odd-graded self-adjoint unbounded operator on $H$ for which

\begin{enumerate}[\rm (i)]

\item The operator $D$ has compact resolvent.

\item The $G$-action on $H$ preserves the domain of $D$, and $D {-} g(D)$ extends to a bounded operator on $H$, for every $g\in G$,  and defines a strongly continuous, locally bounded,  bounded operator-valued function of $g\in G$.

\end{enumerate}

\end{definition}

\begin{lemma}[See \cite{BaajJulg}]
\label{lem-baaj-julg}
If  $(H, D)$ is an unbounded Kasparov cycle for $KK_G(\C,\C)$, and if $F$ is the bounded transform
\[
F= {D}{(1+D^2)^{-\frac12}},
\]
then $(H,F)$ is a bounded  Kasparov cycle for $KK_G(\C,\C)$.
\qed
\end{lemma}

\begin{theorem}[See  {\cite[Theorem 6.1]{Nishikawa19}}] 
\label{thm_gamma} 
Let $(H,D)$ be an unbounded Kasparov cycle for $KK_G(\C,\C)$. Let $B$ be a dense $G$-subalgebra of $C_0(E)$ consisting of compactly supported functions that includes cutoff functions for the $G$-action.   Assume that there is a non-degenerate representation of the $G$-$C^*$-alg\-ebra $C_0(E)$ on $H$ which has the following properties: 
\begin{enumerate}[\rm (i)]

\item  Each $b\in B$  preserves the domain of $D$,  the commutator  $[D,g(b)]$ extends to a bounded operator on $H$, and defines a uniformly bounded, bounded operator-valued function of $g\in G$.

\item For every $b\in B$ there is a compact subset $K\subseteq E$ such that 
\[
\operatorname{Supp} \left ( [D , g(b)]\right ) \subseteq g \cdot K 
\]
for all $g\in G$.

\end{enumerate}
If $F$ is the bounded transform of $D$, then the pair $(H,F)$ is a cycle for $KK_G(\C,\C)$  that satisfies the conditions \textup{(iia)} and \textup{(iib)} of Property $(\gamma)$ in Definition \ref{def_gamma}.
\end{theorem}


\section{The Julg-Valette Cycle}

In this section we shall briefly review the construction  in 
\cite{BGH19}  of    the \emph{Julg-Valette complex}  for a
$\CAT(0)$-cubical space, and the associated \emph{Julg-Valette cycle}
for $KK_G(\C,\C)$. The construction is based on the ideas that Julg
and Valette applied to the case of trees \cite{JulgValetteQ_p}, but the extension to general
$\CAT(0)$-cubical spaces requires some additional care. One specific
issue is the need to equip cubes with  \emph{orientations}, and we shall present here an alternative but equivalent approach to the one in \cite{BGH19} that is better suited to our later arguments.  Otherwise our account here is identical to the one in \cite{BGH19}.
 
We first recall some necessary ideas from the theory of  $\CAT(0)$-cubical spaces.  For more details see \cite{BGH19} and the references given there, especially \cite{NibloReeves98}.  

For the rest of this paper, $X$ will be a bounded geometry
$\CAT(0)$-cubical space, as in \cite[Section 2.2]{NibloReeves98}.  In
particular, $X$ is obtained by identifying cubes, each of which is
isometric to the standard Euclidean cube $[0,1]^q$ of the appropriate
dimension, by isometries of their faces.  We shall refer to $0$-cubes
as vertices, $1$-cubes as edges, etc.  The bounded geometry
condition asserts that there is a uniform bound on the number of edges
that contain a given vertex, and this implies that $X$ is finite
dimensional.  Further, we shall assume all group actions on $X$ are by
automorphism of its structure as a $\CAT(0)$-cubical space.  In
particular, a group acting on $X$ permutes the vertices, the edges,
and the $q$-cubes for each fixed $q$.

The \emph{midplanes} of a $q$-cube $C$ in $X$ correspond to the
intersections of the standard cube with the coordinate hyperplanes
$x_j{=}\frac 12 $ ($j{=}1,\dots, q$) and there are therefore $q$
midplanes of $C$.  We generate an equivalence relation on the
collection of all the midplanes of all the cubes in $X$ by declaring
as equivalent any two midplanes whose intersection is itself a
midplane.  The union of all midplanes in a given class is called a
\emph{hyperplane}, and each hyperplane $H$ satisfies the following
important properties: (i) $H$ is connected, and (ii) the complement
$X \setminus H$ has precisely two path components.  We shall say that
a hyperplane $H$ \emph{separates} two subsets of $X$ if these subsets
lie in distinct components of $X\setminus H$.  A group acting on $X$
necessarily permutes the hyperplanes.  The hyperplanes of $X$
play an important role in what follows.

\begin{example} 
\label{ex-products-and-hyperplanes}
In the case of a tree, the hyperplanes are simply the midpoints of edges.  In the case of a product of two $\CAT(0)$-cubical spaces (for instance, a product of two trees), the hyperplanes are the products of one of the spaces with a hyperplane in the other.
\end{example}

We shall handle the issue of orientations that we mentioned above by using differential forms.  By a \emph{smooth form}  on a cube $C$ we shall always mean a complex-valued $C^\infty$-differential form defined on all of $C$, including its boundary.  So on the cube $[0,1]^q$ the smooth $p$-forms can be written 
\[
\alpha = \sum_I f_I dx_I
\]
where $I = \{ \, i_1< \cdots < i_p\,\}$ is a multi-index, and $f_I$ is a smooth, complex-valued function on  the cube $[0,1]^q$.  The interior of each cube is  a Riemannian manifold of volume one, and so each space $\Omega^p(C)$ of smooth $p$-forms carries a natural inner product.

\begin{definition}
\label{def-A-space}
If $C$ is an individual  cube in $X$, then we shall denote by   $\A (C)$  the one-dimensional Hilbert  space of top-degree differential forms on $C$ with constant coefficients.
\end{definition}

\begin{definition}
\label{def-A-space-of-X}
For $q=0,\dots, \dim(X)$ we shall denote by $\A^q (X)$ the algebraic   direct sum of the spaces $\A(C)$  over all $q$-cubes in $X$.  We shall denote by $\A_{L^2}^q(X)$ the Hilbert space orthogonal direct sum of the  spaces $\A(C)$. 
\end{definition}

We now define a differential
\begin{equation}
\label{eq-jv-differential}
d_{\mathrm{JV}} \colon \A ^q (X ) \longrightarrow \A ^{q+1}(X )
\end{equation}
using the following notion of adjacency.

\begin{definition}
 Let $C$ be a cube in $X$. A hyperplane $H$  is \emph{adjacent to $C$}   if it  is disjoint from $C$, but  intersects a cube   that includes $C$ as a codimension-one face. \end{definition}

\begin{example} 
In the case of a tree, each hyperplane is adjacent to precisely two vertices, namely the endpoints of the edge in which the hyperplane is included.
\end{example}

At this point we need to choose a \emph{base vertex} $P$ in $X$, which will remain fixed throughout the rest of the paper (by a \emph{vertex}  we mean  a $0$-cube of $X$). 

\begin{definition}
\label{def-coordinate-for-a-hyperplane}
Let $H$ be a hyperplane and let $D$ be a cube in $X$ that intersects $H$ (necessarily in a midplane of $D$). Define a coordinate function
\[
x_H \colon D \longrightarrow \R
\]
as follows: 
\begin{enumerate}[\rm (i)]

\item $x_H$ is an affine function on $D$

\item $x_H$ assumes the value $1$ on the codimension-one face of $D$ that is disjoint from $H$ and separated by $H$ from the base vertex, while it  assumes the value $0$ on the opposite codimension-one face, disjoint from $H$ but not separated by $H$ from the base vertex.

\end{enumerate}
Then define 
\[
\alpha_H = d x_H .
\]
This is a constant-coefficient one-form on $D$.
\end{definition}

\begin{definition}
\label{def-weight-function0}
A \emph{weight function} for $X$ is a positive function on the set of hyperplanes.
\end{definition}

\begin{definition}
\label{def-weight-function-proper}
A weight function $w$ is \emph{proper} if the set
$\{ \, H\,  : \,  w(H)\leq M \, \}$ is finite for any $M>0$.  
\end{definition}

\begin{definition}
\label{def-G-adapted-weight-function}
Let $G$ be a second countable and locally compact group acting on the
$\CAT(0)$-cubical space $X$  by automorphisms. A weight function $w$
is \emph{$G$-adapted} if 
\[
\sup_{H}\bigl |w(H) - w(gH)\bigr | < \infty
\]
for every $g\in G$, and if  there is an open subgroup of $G$ such that
$w(gH)=w(H)$ for  every $H$ and every element $g$ of the subgroup. 
\end{definition}

\begin{remark}
  Of course, if $G$ is 
  discrete the final condition is automatic.
  In general it would be sufficient to require
  that $g\mapsto w(gH)$ is continuous for every $H$, but the stronger
  condition we have chosen is more convenient.  
\end{remark}

In what follows, our preferred weight function will be
\begin{equation}
\label{eq-distance-weight-function}
w(H)=\operatorname{distance}(P, H),
\end{equation}
which is both $G$-adapted and   proper, if the complex $X$ has bounded geometry (see Section 1).  However, the construction that follows can be carried out for a general weight function, and it will be helpful for us not to fix a weight function until the final section of the paper.  

\begin{definition}
\label{def-jv-differential}
If $\beta \in \A (C)$, then we define 
\begin{equation}
\label{eq-def-of-d}
d_{\mathrm{JV}} \colon \beta \longrightarrow \sum_{H\in \SAH(C)} w(H) \cdot  \alpha _H \wedge \beta
\end{equation}
where:
\begin{enumerate}[\rm (i)]

\item 
The sum is over the set $\SAH(C)$ of all hyperplanes that are adjacent to $C$  and   separate $C$ from the base vertex.

\item 
On the right-hand side of the formula, $\beta$ is pulled back from $C$ to the cube $D_H$ that is intersected by $H$ and that contains $C$ as a codimension one face (via the orthogonal projection from $D_H$ to its face $C$).  Thus the summands in the formula above lie in different direct summands of $A^{q+1}(X)=\oplus_{\dim(D)=q+1} A (D)$.

\end{enumerate}

\end{definition}

It is shown in \cite{BGH19} that $d_{\mathrm{JV}}$ is indeed a
differential, and that the cohomology of the \emph{Julg-Valette
  complex} 
\begin{equation}
\label{eq-jv-complex}
\xymatrix{
  \A^0 (X)\ar[r]^{d_{\mathrm{JV}}} & \A^1(X)\ar[r]^{d_{\mathrm{JV}}}
  & \cdots & \ar[r]^{d_{\mathrm{JV}}}
  & \A^N(X) 
}, \quad N=\dim(X)
\end{equation}
vanishes in positive degrees, and is one-dimensional in degree zero.

\begin{example}
  If $X$ is a tree, then for each $0$-cube except the base vertex
  there is a unique hyperplane that is adjacent to the cube and
  separates it from the base vertex.  If we set $w(H)\equiv 1$ then
  the single differential in the Julg-Valette complex identifies with
  operator studied by Julg and Valette \cite{JulgValetteQ_p}.
\end{example}

\begin{example}
If $X$ is a product of two $\CAT(0)$-cubical spaces (for instance, a product of trees), then the Julg-Valette complex for $X$ identifies with the tensor product of the Julg-Valette complexes for the factors.  Note that in this case a pair of weight functions, one on each factor, determines a weight function on the product; see Example~\ref{ex-products-and-hyperplanes}.
\end{example}
The Julg-Valette differential in Definition~\ref{def-jv-differential} has a formal adjoint 
\[
\delta_{\mathrm{JV}} \colon \A^{q+1}(X) \longrightarrow \A ^q (X)
\]
with respect to the inner products of Definition~\ref{def-A-space-of-X}, given by the formula
\begin{equation}
\label{eq-def-of-delta}
\delta_{\mathrm{JV}}\colon \beta 
\longmapsto 
\sum_{H\in \Mid (C) }w(H) \cdot \iota_H (\beta) ,
\end{equation}
for $\beta\in A(C)$, where:
\begin{enumerate}[\rm (i)]

\item 
 The sum is over the set $\Mid (C)$  of hyperplanes that cut through $C$ (that is, they include  a midplane from $C$).

\item  On the right-hand side, $\iota_H(\beta)$ is contraction of $\beta$ by the coordinate vector field dual to the coordinate $x_H$, followed by restriction to the face of $C$ where $x_H = 1$.
\end{enumerate}
See \cite[Prop.\,3.19]{BGH19}.  

\begin{definition}
\label{def-julg-valette-op}
The \emph{Julg-Valette operator} is the unbounded operator 
\begin{equation*}
D _{\mathrm{JV}} = d_{\mathrm{JV}} + \delta_{\mathrm{JV}}  
\end{equation*}
on the Hilbert space 
\[
\A^*_{L^2}(X) = \bigoplus_q  \A^q_{L^2}(X),
\]
with domain 
\[
\A^*(X) = \bigoplus _{q} \, \A^q(X) .
\]
\end{definition}

If $\beta\in A(C)$, then it is not difficult to compute that
\begin{equation}
\label{eq-fmla-for-D-squared}
D^2 _{\mathrm{JV} } \beta =  \sum _{H \in \SAH(C)} w(H)^2 \beta + \sum_{H \in \Mid (C)} w(H)^2\beta 
\end{equation}
See \cite[Prop.\,3.25]{BGH19}.   The following result is an easy consequence:

\begin{lemma}[{\cite[Lemmas 9.4 and 9.5]{BGH19}}]
\label{lem-JV-compact-resolvent}
If $w$ is a proper  weight function, then   $D_{\mathrm{JV}}$ is an essentially self-adjoint Hilbert space operator with compact resolvent. \qed
\end{lemma}

The Julg-Valette operator is also nearly equivariant in the sense appropriate to equivariant $KK$-theory:

\begin{lemma} [{\cite[Lemma~9.6]{BGH19}}]
\label{JV-bounded-G-translates}
If a group $G$ acts by automorphisms of $X$, and if the weight function is $G$-adapted, then for every   $g\in G$ the operator 
$D_{\mathrm{JV}} {-} g(D_{\mathrm{JV}})$ extends to a bounded operator on $H$,  and the formula 
\[
g \longmapsto D_{\mathrm{JV}} {-} g(D_{\mathrm{JV}})
\]
defines a strongly continuous, locally bounded,  bounded operator-valued function on $G$. \qed
\end{lemma}

\begin{remark}
\label{rem-continuity-not-an-issue}
The continuity and local boundedness of the function in the lemma are automatic, since in fact the function is locally constant: if $g$ lies in a sufficiently small open subgroup of $G$, then $g(D_{\mathrm{JV}}) = D_{\mathrm{JV}}$.
\end{remark}

Therefore,  given any proper and $G$-adpated weight function, we obtain  an unbounded Kasparov cycle for $ KK_G (\C,\C)$ in the sense of Definition~\ref{def-unbounded-cycle}. It is actually independent of the choice of weight function thanks to the homotopy invariance of $KK$-theory, and we shall use the notation
\[
[D_{\mathrm{JV}}]\in KK_G (\C,\C) .
\]
The following is the main technical result from the paper \cite{BGH19}. 

\begin{theorem}[{\cite[Theorem 9.14]{BGH19}}]
\label{thm-JV=1}
Assume that $X$ is a bounded geometry \textup{(}and finite
dimensional\textup{)} $\CAT(0)$-cubical space, and that a locally
compact, second-countable group $G$ acts on $X$ by
automorphisms.  Then
$[D_{\mathrm{JV}}] =1_G \in KK_G (\C,\C)$. \qed
\end{theorem}

\begin{remark}
Actually, Theorem 9.14 in \cite{BGH19} proves that  $[D_{\mathrm{JV}}]$  is equal to the Pytlik-Szwarc element $[D_{\mathrm{PS}}]$ discussed in Section~1 of this paper.   But as we noted there, the Pytlik-Szwarc element is in turn equal to the identity element $1_G$ for elementary reasons.
\end{remark}


\section{A De Rham-Type Complex}

In this section we shall use differential forms and the de Rham operator to construct a new unbounded  Kasparov cycle $D_{\mathrm{dR}}$
for the Kasparov group $KK_G (\C,\C)$.  In the next section we shall show that 
$[D_{\mathrm{dR}}] = [D_{\mathrm{JV}}]$ and that, in the case of a
proper and cocompact action, the bounded transform of $D_{\mathrm{dR}}$ has Property $(\gamma)$.
 
We  begin by reviewing a few simple facts about the de Rham differential on a  single cube $C$ in $X$. 

\begin{definition}
For $p{\ge} 0$, we shall denote by $\Omega^p(C)$ the space of smooth differential $p$-forms on $C$, by  $\Omega_{L^2}^{p}(C)$ the Hilbert space of $L^2$-differential $p$-forms on $C$, by $\Omega^p_0(C)\subseteq \Omega^p(C)$ the space of those smooth $p$-forms on $C$ that pull back to zero on each open face of $C$, and by $\Omega^p_{00}(C)$ the space of smooth forms that are compactly supported in the interior of $C$.
\end{definition}

We shall consider the de Rham differential $d$ as an unbounded Hilbert space operator from $\Omega^p_{L^2}(C)$ to $\Omega^{p+1}_{L^2}(C)$ with domain $\Omega_0^p(C)$.  If we denote  by  $d^\diamond$ the  formal adjoint of $d$ in the sense of partial differential equations, so that 
\begin{equation}
\label{eq-adjoint-relation}
\langle d \sigma, \tau \rangle = \langle \sigma, d^\diamond\tau \rangle
\end{equation}
for all $\sigma\in \Omega^p_{00}(C)$ and all $\tau\in \Omega^{p+1}(C)$, then in fact the relation \eqref{eq-adjoint-relation} holds for all $\sigma\in \Omega^p_{0}(C)$ and all $\tau\in \Omega^{p+1}(C) $.  Indeed, if we choose an orientation on $C$ and introduce the associated Hodge $\star$-operator, then for all $\sigma \in \Omega^p(C)$ and all $\tau\in \Omega^{p+1}(C)$,
\[
\langle d \sigma, \tau \rangle - \langle\sigma, d^\diamond\tau \rangle
=
\int_{\partial C} \overline \sigma\wedge \star  \tau .
\]
Henceforth we shall consider $d^{\diamond}$ as an unbounded Hilbert space operator with domain $\Omega^{p+1}(C)$. Form  the direct sums 
\[
\Omega^* _0 (C) = \bigoplus_p  \Omega^p_0(C)
\quad \text{and} \quad
\Omega^* _{L^2}(C) = \bigoplus_p  \Omega^p_{L^2}(C) ,
\]
 and then form the unbounded operator
\begin{equation}
\label{eq-de-Rham-op-1}
D = d+d^\diamond \colon \Omega^* _{L^2} (C) \longrightarrow \Omega^*_{L^2} (C)
\end{equation}
whose domain is the intersection of the domains of $d$ and $d^\diamond$, which is $\Omega^*_0(C)$. This  is a symmetric operator.

\begin{lemma}
\label{lem-ess-self-adj-1}
The operator $D$ in 
\eqref{eq-de-Rham-op-1} is essentially self-adjoint, and the self-adjoint closure of $D$ has compact resolvent.  The kernel of the closure is precisely the one-dimensional space $\A(C)$ of top-degree constant coefficient forms.
\end{lemma}

\begin{proof}
Let $q$ be the dimension of $C$ and let $H_1,\dots, H_q$ be the hyperplanes that intersect $C$.  Let $k_1,\dots, k_q$ be nonnegative integers, and form the span of all differential forms of the type 
\[
\sigma_1\wedge\cdots \wedge \sigma_q,
\]
 where each $\sigma_j$ is either $\sin (\pi k_j x_{H_j})$ or $\cos(\pi k_j x_{H_j})d_{X_j}$.  The span is a finite-dimensional subspace of $\Omega^*_0(C)$; it is invariant under $D$; the square of $D$ on the span is the scalar multiple
\[
\pi^2 k_1^2+ \cdots + \pi^2 k_q^2
\]
of the identity operator; and the family of all such spans (as the list of integers $k_1,\dots, k_q$ varies) is an orthogonal  decomposition of $\Omega^*_{L^2} (C)$.  The lemma follows from these facts.
\end{proof}

Let us now fix, for the rest of this section, a proper and $G$-adapted weight function.  For each cube $C$ in $X$ define an affine function
$
w_C\colon C \to \R
$
by the formula 
\begin{equation}
\label{eq-def-wc}
w_C = \sum _{H\in \Mid (C)} w(H) x_H \;, 
\end{equation}
and form  the Witten-type perturbation
\[
d_w = e^{-w_C}d e^{w_C} \colon \Omega^*_0(C) \longrightarrow \Omega^*(C) .
\]
It follows from the definition   that 
\begin{equation}
\label{eq-perturbed-de-rham}
d_{ w}\colon \beta \longmapsto  d\beta  +  \sum_{H\in \Mid (C)}w(H)dx_H \wedge \beta  ,
\end{equation}
so that $d_w$ is in fact a bounded perturbation of $d$. Let $d_w^\diamond$ be its formal adjoint, considered as an unbounded operator with  domain $\Omega^*(C)$, and then form the Witten-type de Rham operator 
\begin{equation}
\label{eq-def-witten-de-rham}
D_w = d_w + d_w^\diamond \colon \Omega_{L^2}^*(C) \longrightarrow \Omega_{L^2}^*(C)
\end{equation}
with the same domain as $D$, namely $\Omega^*_0(C)$.  Being a bounded perturbation of an essentially self-adjoint operator with compact resolvent,  the operator $D_w$ is itself essentially self-adjoint with compact resolvent; see   \cite[Ch.\ 4 Theorem\ 3.1.11]{Kato95} and \cite[Ch.\ 5, Theorem\ 4.4.4]{Kato95}.

We are ready now to proceed from a single cube to the cubical space $X$.

 \begin{definition}
 \label{def-omega-and-omega-zero}
Denote by $\Omega^*_{L^2}(X)$ the Hilbert space direct sum
\[
\Omega^{\ast}_{L^2}(X) = \bigoplus_{C} \Omega_{L^2}^{*}(C)
\]
over  all the cubes in $X$ of all dimensions, and denote by 
\[
\Omega^{\ast}_{0}(X) \subseteq  \Omega_{L^2}^{*}(X)
\]
the algebraic direct sum of the spaces $\Omega^*_0(C)$.
\end{definition}

 Our general aim is to assemble the operators  $D_w$   in \eqref{eq-def-witten-de-rham} into one essentially self-adjoint    operator with compact resolvent on the full space $\Omega^*_{L^2}(X)$.  However this requires some care, as the following computation shows.

   Let  $C$ be a one-dimensional cube and let $H$ be the hyperplane that intersects $C$.  For each integer $k > 0$ the span  in $\Omega^*_0(C)$  of the forms 
\[
\sin (\pi k x_H) \quad \text{and} \quad 
 d_w \sin (\pi k x_H) = \pi k \cos(\pi k  x_{H})d_{x_{H}} {+} w(H) \sin (\pi k x_{H}) d_{x_{H}} 
\]
is invariant under the action of $D_w$, and furthermore  
\[
D^2_w = \pi ^2 k^2 + w(H)^2
\]
on the span.  The spans for distinct $k$ are orthogonal to one another, and their mutual orthogonal complement in $\Omega^*_{L^2} (C)$ is spanned by
\[
e^{w(H)x_H} dx_H\in \Omega^*_0(C),
\]
which lies in the kernel of $D_w$, and indeed spans it.  It follows that
\[
\operatorname{Spectrum} (D_w^2) = 
 \{\,  0\,  \} \,\, \sqcup \,\,  \{ \, \pi^2 k^2 + w(H) ^2\, : \, k =1,2\dots \,  \} 
\]
(we mean here the spectrum of the square of the self-adjoint extension of $D_w$). 

We find that in the case of a tree, for example,  the operator $D_w$ on each cube has a one-dimensional kernel, and in order to construct a single Witten-type de Rham operator on the tree we need to ``cancel'' these kernels against one another so as to obtain an essentially self-adjoint, compact resolvent operator on $\Omega^*_{L^2}(X)$. The next two definitions are dedicated to solving this problem (for $\CAT(0)$-cubical spaces in all dimensions).

\begin{definition} 
\label{def-e-differential}
If $D$ is any cube in $X$, and if $H$ is any hyperplane that intersects $D$, then let   
\[
y_H = x_H -1  \colon D \longrightarrow \R ,
\]
  where $x_H$ is as in Definition~\ref{def-coordinate-for-a-hyperplane}.  Observe that 
\[
dy_H = \alpha _H = dx_H. 
\]
Define a linear operator 
\[
e_{w}\colon \Omega^p(X) \longrightarrow  \Omega^{p+1}(X)
\]
by the formula 
\[
e_w\colon \beta  \longmapsto \sum_{H\in \SAH(C)} w(H )\cdot e^{w(H)y_H}\cdot \alpha _H \wedge \beta
\]
for $\beta \in \Omega^*(C)$, where:
\begin{enumerate}[\rm (i)]

\item As in Definition~\ref{def-jv-differential},  $\SAH(C)$ is the set of all hyperplanes that are adjacent to $C$  and   separate $C$ from the base vertex.

\item
For a hyperplane $H\in \SAH(C)$, if $D_H$ is the cube that is intersected by $H$, and that includes $C$ as a codimension-one face, then we view $\alpha_H\wedge \beta$ as a differential form on  $D_H$ by pulling back $\beta$ along the orthogonal projection from $D_H$ to  $C$.
\end{enumerate}

 \end{definition}

\begin{definition} 
\label{def-de-Rham-op}
We shall denote by  $D_{\mathrm {dR}}$ the following symmetric operator on $\Omega^{\ast}_{L^2}(X)$ with domain $\Omega^*_0(X)$:
\[
 D_{ \mathrm{dR}} = d_{ w} + e_{w} + d_{w}^\diamond + e_{ w}^\diamond .
\]
By $d_w$ we mean here the direct sum of all the Witten-type operators \eqref{eq-perturbed-de-rham} over all the cubes of $X$, and the same for the formal adjoint $d_w^\diamond$.
\end{definition}
 
We shall prove the following result:

\begin{theorem}
\label{thm-de-Rham-compact-resolvent-1}
If the weight function $w$ is proper, then the operator $D_{\mathrm{dR}}$ is essentially self-adjoint and has compact resolvent.
\end{theorem}

For the purposes of the proof, and for other purposes as well, it will be convenient to decompose $\Omega_{L^2}^*(X)$ into  a direct sum of ``blocks'' in such a way that the operator $D_{dR}$ is block-diagonal.

\begin{definition}
  Let  $Q$ be a vertex of $X$. We shall say that a cube $C$
  \emph{leads $Q$ to the base vertex $P$} if the following conditions hold:  
\begin{enumerate}[\rm (i)]

\item The cube $C$ includes the vertex $Q$.

\item All the hyperplanes that intersect $C$ also separate $Q$ from
  $P$.   That is, $\Mid(C)\subseteq \SAH(Q)$. 

\end{enumerate}
We shall denote by 
\[
\Omega^{\ast}_{L^2}(X)_Q \subseteq \Omega^*_{L^2}(X)
\]
  the direct sum  of the spaces   $\Omega^*_{L^2}(C)$ over those cubes   $C$ that lead $Q$ to $P$.  We define $\Omega^*_0(X)_Q$ and $\Omega^*(X)_Q$ similarly.
\end{definition}

\begin{remark} 
\label{rem-normal-cube-path}
For any vertex $Q$, the collection $\SAH(Q)$ coincides with collection $\Mid(D)$ for  a distinguished cube $D$  in $X$, namely the first cube in the normal cube path from $Q$ to $P$ in the sense of \cite[Sec.\ 3]{NibloReeves98}.  The cubes in $X$ that lead $Q$ to $P$ are the faces of  $D$ that include the vertex $Q$. The union of all such cubes  is isometric to a product $(0,1]^q$ of half-open intervals, in such a way that  the coordinates $x_H$   are the standard coordinates on $(0,1]^q$.
\end{remark}

 We have the following Hilbert space direct sum decomposition:
\begin{equation}
\label{eq-normal-cube-decomp}
\Omega^{\ast}_{L^2}(X)=\bigoplus_{Q} \Omega^{\ast}_{L^2}(X)_Q.
\end{equation}
This is because  each cube $C$  in $X$ has a unique vertex $Q$  that is separated from the base vertex by all the hyperplanes intersecting $C$.  Inspecting the formulas, we see immediately that: 

\begin{lemma} 
The operator $D_{\mathrm{dR}}$ is block-diagonal with respect to the Hilbert space decomposition \eqref{eq-normal-cube-decomp}.
\qed 
\end{lemma}

To begin the proof of Theorem~\ref{thm-de-Rham-compact-resolvent-1}, let  us examine   the   case of a vertex $Q$ for which $\SAH(Q)$ consists of a single hyperplane $H$, so that 
\begin{equation}
\label{eq-simple-summand}
\Omega^*_0(X)_Q = \Omega_0^*(Q) \oplus \Omega^*_0(E),
\end{equation}
where the edge $E$ leads $Q$ toward $P$.  A little more generally, let  $Q$ be any vertex and let $E$ be any edge leading   $Q$ to $P$, and consider the operator 
\begin{equation}
\label{eq-interval-operator}
\begin{bmatrix}
0 & e^\diamond \\ e & d_w{+}d_w^\diamond
\end{bmatrix}
\colon 
 \Omega_{L^2}^*(Q) \oplus \Omega^*_{L^2}(E)
\longrightarrow 
 \Omega_{L^2}^*(Q) \oplus \Omega^*_{L^2}(E)
\end{equation}
with domain $\Omega_0^*(Q) \oplus \Omega^*_0(E)$, where $e$ acts as 
\[
\Omega^0_0(Q) \ni 1 \longmapsto w(H) e^{w(H)y_H} dx_H \in \Omega^1_0(E).
\]
In the special case described by  \eqref{eq-simple-summand}, this is the restriction of the operator $D_{\mathrm{dR}}$ to $\Omega^*_{L^2}(X)_Q$.
The images of symmetric operators
\[
\begin{bmatrix}
0 & e^\diamond \\ e & 0 
\end{bmatrix}
\quad \text{and}\quad 
\begin{bmatrix}
0 & 0 \\ 0 & d_w{+}d_w^\diamond
\end{bmatrix}
\]
are orthogonal to one another, and indeed the image of the first is precisely the kernel of the second.  The kernel   is spanned by 
\[
1\in \Omega^0_0(Q) \quad \text{and} \quad e^{w(H)x_H} dx_H \in \Omega^1_0(E),
\]
and the first operator not only leaves this space invariant, but its square there is 
\[
w(H)^2 \cdot \| e^{-y_H}\|^2 _{L^2 (0,1)}
	= \tfrac 12 w(H)(1 - e^{-2w(H)}) .
\]
As a result there is a Hilbert space orthonormal basis consisting of eigenvectors in the domain of \eqref{eq-interval-operator}, and the squares of the eigenvalues  are the scalars in the list 
\begin{equation}
\label{eq-eigenvalues-of-dR-squared}
\bigl \{ \, \tfrac 12 w(H)(1 - e^{-2w(H)}) \,\bigr  \} 
\,\,\sqcup\,\, 
\bigl \{ \, \pi^2 k^2 + w(H) ^2\, : \, k =1,2\dots \,\bigr  \} .
\end{equation}
Let us note for later use that 
\begin{equation}
\label{eq-simple-inequality}
w(H)^2 \ge \tfrac 12 w(H)(1 - e^{-2w(H)}) 
\end{equation}
no matter what the value of $w(H){>}0$, so that the least eigenvalue of the square is always the first scalar in the list \eqref{eq-eigenvalues-of-dR-squared}.

\begin{proof}[Proof of Theorem~\ref{thm-de-Rham-compact-resolvent-1}]
Fix a  vertex $Q$ and let  $H_1,\dots, H_q$ be  the hyperplanes adjacent to $Q$ and separating $Q$ from $P$.  Denote by $E_j$ the unique edge leading $Q$ to $P$ that is cut by $H_j$.  There is then a unitary  isomorphism of graded Hilbert spaces
\begin{equation}
\label{eq-pieces-decomp-hill2}
   \underset{j}{\widehat{\bigotimes}}\Bigl  ( \Omega_{L^2}^*(Q)\oplus \Omega^*_{L^2}(E_j)   \Bigr )
\stackrel\cong \longrightarrow \Omega^*_{L^2}(X)_Q 
\end{equation}
given by exterior multiplication of forms, as follows.  If $C$ is a cube that leads $Q$ to $P$, then define a map
\[
\Omega^*_{L^2}(Q)\oplus \Omega^*_{L^2}(E_j)
\longrightarrow \Omega^*_{L^2}(C)
\]
by pulling back forms from $E_j$ to $C$ along the orthogonal projection, when $E_j$ is included in $C$, and by pulling back from the point $Q$ when $E_j$ is not included in $C$.
The morphism \eqref{eq-pieces-decomp-hill2}  is defined by  pulling back forms in this way, then forming the exterior products of the pullbacks.

 The isomorphism \eqref{eq-pieces-decomp-hill2} restricts to an inclusion 
\begin{equation}
\label{eq-pieces-decomp-hill3}
   \underset{j}{\widehat{\bigotimes}}\Bigl  ( \Omega_{0}^*(Q)\oplus \Omega^*_{0}(E_j)   \Bigr )
\stackrel\cong \longrightarrow \Omega^*_{0}(X)_Q 
\end{equation}
Denote by 
\[
D_{j} \colon  \Omega_{0}^*(Q)\oplus \Omega^*_{0}(E_j)  \longrightarrow \Omega^*(Q) \oplus \Omega^*(E_j)
\]
  the operator in \eqref{eq-simple-summand}.
The operator $D_{\mathrm{dR}}$ on $\Omega^*(X)_Q$ identifies, via \eqref{eq-pieces-decomp-hill2}, with the sum 
\[
\sum _j I\widehat \otimes   \cdots \widehat \otimes I  \widehat \otimes D_{j}\widehat \otimes I \widehat \otimes \cdots \widehat \otimes I ,
\]
from which we see that  there is  an orthornormal basis for  $\Omega^*_{L^2}(X)_Q$ consisting of eigenvectors for   $D_{\mathrm{dR}}$ in $\Omega^*_{0}(X)_Q$, and indeed in the image of \eqref{eq-pieces-decomp-hill3}, for which in addition the squares of eigenvalues are sums of one term from each of the   lists  
\begin{equation*}
\bigl \{ \, \tfrac 12 w(H_j)(1 - e^{-2w(H_j)}) \, \bigr \} 
\,\,\sqcup\,\, 
\bigl \{ \, \pi^2 k^2 + w(H_j) ^2\, : \, k =1,2\dots \, \bigr \} ,
\end{equation*}
for $j=1,\dots q$.  This proves the theorem.
\end{proof}

The proof above, together with \eqref{eq-simple-inequality}, also yields the following estimate, which we shall use in the next section.

\begin{lemma} \label{lemma_op} 
For any vertex $Q$ and for any $\beta\in \Omega^*_0(X)_Q$,  
\[
\pushQED{\qed} 
\| D_{\mathrm{dR}} \beta \|^2   
	\geq \sum_{H\in \SAH(Q)}   
	\tfrac12 {w(H)} (1-e^{-2w(H)}) \| \beta\|^2 .
\qedhere
\] 
\end{lemma}

\begin{lemma} [Compare {\cite[Lemma~9.6]{BGH19}}.]
\label{lem-G-bounded-property-for-de-Rham}
If the weight function $w$ is $G$-adapted, then  for any $g$ in $G$, the difference $g(D_{\mathrm{dR}}) - D_{\mathrm{dR}}$ is a bounded operator.  Moreover $ \| g(D _{\mathrm{dR}}) - D _{\mathrm{dR}} \| $ is bounded by a constant times
\begin{multline*}
 \sup \bigl \{ \, \bigl |gw(H)-w(H) \bigr |  :
\text{\rm $H$ is a hyperplane of $X$} \, \bigr \} 
\\
+ \max \bigl \{ \,  w(H) : \text{\rm $H$ separates $P$ and $gP$}\, \bigr \}  ,
\end{multline*}
where the constant is independent of $g$ and the choice of weight function $w$.
\end{lemma}

\begin{proof} 
For the purposes of this proof it  will be convenient to make explicit note of the dependence of $D_{\mathrm{dR}}$ on the choice of basepoint, and on the choice of weight function.  So we shall write 
\[
D_{\mathrm{dR}} = D_{w, P} ,
\]
so that 
\[
g(D_{\mathrm{dR}}) = D_{gw,gP},
\]
where 
$gw(H) = w(g^{-1}H)$.  We shall bound   the two differences 
\[
D_{gw, gP} - D_{w,gP} \quad \text{and} \quad 
D_{w,gP} - D_{w,P}
\]
separately to obtain the estimate in the statement of the lemma.

For the first, let us write 
\begin{multline*}
D_{gw, gP} - D_{w,gP} 
= (d_{gw,gP} - d_{w,gP}) +  (e_{gw,gP} - e_{w,gP})
\\
+( d_{gw,gP}^\diamond - d_{w,gP}^\diamond   ) + (e^\diamond_{w,P} - e^\diamond_{gw,gP})
\end{multline*}
Since $w$ is $G$-adapted, there is some $M\ge 0$ for which 
\[
\bigl |gw(H)-w(H) \bigr | \leq M \qquad \forall H.
\]
It follows from the formula for $d_w$ that 
\[
\|d_{gw,gP} - d_{w,gP} \| \le \dim (X)\cdot M ,
\]
and of course from this we also get 
\[
\| d_{gw,gP}^\diamond - d_{w,gP}^\diamond\|
\le 
\dim (X)\cdot M .
\]
Next,   if $x,w\ge 0$, and if $M\ge m\ge 0$, then
\begin{equation}
\label{eq-calculus-estimate}
\bigl | we^{-wx} - (w+m)e^{-(w+m)x}\bigr | \le 2M ,
\end{equation}
since for instance one can write  
\[
\begin{aligned}
\bigl | we^{-wx} - (w+m)e^{-(w+m)x}\bigr  |
	& \le m \Bigl  | w e^{-wx} \frac{(1 - e^{-mx})}{m} \Bigr | + m \\
	&\le m \left  | e^{-wx} wx \right | + m ,
\end{aligned}
\]
from which  \eqref{eq-calculus-estimate} follows easily.   From \eqref{eq-calculus-estimate}  and the formula for $e_{w,P}$ in Definition~\ref{def-e-differential} we obtain  
\[
\| e_{gw, gP} - e_{w,gP} \| \le 2 \dim (X) M ,
\]
and hence also
\[
\| e_{gw, gP}^ \diamond - e_{w,gP} ^ \diamond\| \le 2 \dim (X) M .
\]
Putting everything together, we obtain 
\[
\| D_{gw, gP} - D_{w,gP} \| \le 6 \dim (X) M .
\]

Next, let us turn to the difference $D_{w,gP} - D_{w,P}$.   We shall analyze it part by part, more or less as we did above.  To begin, 
\[
d_{w,gP} \beta - d_{w,P} \beta = 
\sum_{H \in \Mid (C)} 
	w(H) \bigl (\alpha_{H,gP} - \alpha_{H,P} \bigr )
	\wedge \beta 
\]
The difference $\alpha_{H,gP} - \alpha_{H,P}$ is zero unless the hyperplane $H$ separates $P$ and $gP$, in which case the difference is $2 \alpha_{H,gP}$.  So $\| d_{w,gP}  - d_{w,P} \|$ is bounded by a constant times the largest value of $w(H)$ among the hyperplanes separating $P$ and $gP$, as required, and in addition the same is true of $\| d_{w,gP} ^\diamond - d_{w,P} ^\diamond\|$.

The remaining terms in $D_{w,gP} - D_{w,P}$ are treated in a similar way.  Since
\begin{multline*}
e_{w,gP}\beta  - e_{w,P} \beta 
	\\
= \sum w(H) \left (  e ^{-w(H) y_{H,P}} \alpha _{H,P} {\wedge} \beta\, - \,e^{-w(H) y_{H,gP}} \alpha_{H,gP}{\wedge} \beta \right ) ,
\end{multline*}
and since again the summands are zero except when $H$ separates $P$ and $gP$, we can again bound the sum by a constant times the largest value of $w(H)$ among the hyperplanes separating $P$ and $gP$, and of course we can do the same for $e_{w,gP}^\diamond - e_{w,P}^\diamond$, too.  
\end{proof}

Now let us equip the Hilbert space  $\Omega^{\ast}_{L^2}(X)$ with the
$\Z/2$-grading it acquires from the differential form degree.  The
operator $D_{\mathrm{dR}}$ is odd with respect to this grading.
Combined, Theorem~\ref{thm-de-Rham-compact-resolvent-1} and
Lemma~\ref{lem-G-bounded-property-for-de-Rham} give the following
result.  

\begin{theorem}
\label{thm-de-Rham-cycle}
The operator $D_{\mathrm{dR}}$ on $\Omega^*_{L^2}(X)$, constructed
from the weight function \eqref{eq-distance-weight-function}, defines
an unbounded cycle for the group $KK_G(\C,\C)$.  \qed
\end{theorem}

\begin{definition} 
The \emph{de Rham cycle} for the Kasparov group $KK_G(\C,\C)$ is $(\Omega^{\ast}_{L^2}(X), D_{\mathrm{dR}})$, as in Theorem~\ref{thm-de-Rham-cycle}.  We shall denote by  $[D_{\mathrm{dR}}]\in KK_G(\C,\C)$ the associated $KK$-class.
\end{definition}

\section{Proofs of the Main Results}

In this section we shall prove the following two theorems:

\begin{theorem} 
\label{thm-de-Rham-is-JV}
If $G$ acts on $X$ by automorphisms, the de Rham and Julg-Valette
classes are equal to one another in $KK_G(\C,\C)$. 
\end{theorem}

\begin{theorem} 
\label{thm-de-Rham-has-gamma}
If $G$ acts properly and cocompactly on $X$ by automorphisms, the de
Rham cycle has Property $(\gamma)$. 
\end{theorem}

We shall prove Theorem~\ref{thm-de-Rham-is-JV} by constructing an
explicit homotopy between the Julg-Valette and de Rham cycles, or in
other words we shall construct an unbounded cycle for
$KK_G(\C,C[0,1])$ that restricts to the Julg-Valette and de Rham
cycles at $t{=}0$ and $t{=}1$, respectively.

The underlying Hilbert $C[0,1]$-module for the homotopy will be  
%
\begin{equation}
\label{eq-def-curly-H}
\mathcal{H}(X) = \left\{\,
\beta \colon [0,1] \to \Omega^{\ast}_{L^2}(X) \, \middle\vert  \,
        \begin{tabular}{l}$\beta$ is norm-continuous, \\
                  and $\beta_0 \in \A^{\ast}_{L^2}(X)$ \end{tabular} 
\, \right\} .
\end{equation}
Here $ \A^{\ast}_{L^2}(X)$ is the Hilbert space for the Julg-Valette
cycle from Definition~\ref{def-A-space-of-X}, which we note is a
Hilbert subspace of $\Omega^*_{L^2} (X)$.   

Let $w$ be a proper and $G$-adapted weight function. The operator $\D$ for our cycle is  defined by
\begin{equation}
\label{eq-def-curly-D}
\D   \beta  \colon s \longmapsto
\begin{cases}
s^{-1} D_{\mathrm{dR},s} \beta_s  & s\in (0,1] \\
  D_{\mathrm{JV}} \beta _0 & s=0 ,
  \end{cases}
\end{equation}
where:
\begin{enumerate}[\rm (i)]

\item  $D_{\mathrm{dR},s}$ is the de Rham operator from Definition~\ref{def-de-Rham-op}  associated to the weight function $H\mapsto s\cdot w (H)$.

\item  $D_{\mathrm{JV}}$ is the Julg-Valette operator from Definition~\ref{def-julg-valette-op} associated to the weight function $w$ itself.
\end{enumerate}
Of course we need to specify the domain of $\D$. But before we do that, let us carry out some preliminary computations that should help  make Theorem~\ref{thm-de-Rham-is-JV} plausible, and help justify our choice for the operator $\D$.

\begin{definition}
Let $C$ be a $q$-cube in $X$, and let  $H_1,\dots, H_q$ be the hyperplanes that cut $C$ (that is, the hyperplanes in $\Mid(C)$). 
We shall denote by  
\[
 A_s (C)\subseteq \Omega_0^q (C)
 \]
  the one dimensional subspace that is spanned by the form 
\[
 e^{s w_C}dx_{H_1} \cdots dx_{H_q}  .
\]
\end{definition}

\begin{lemma}
\label{lem-invariance-of-A-s1}
The space  $A_s(C)$  is the kernel of the operator
\[
d_{sw} {+}d_{sw}^\diamond
\colon \Omega^*_0 (C) \longrightarrow \Omega^*_0(C) ,
\]
 and of its self-adjoint extension. 
 \end{lemma}
 
 \begin{proof}
 This follows from the computations after Definition~\ref{def-omega-and-omega-zero}.
 \end{proof}

\begin{definition}
We shall denote by $A_s^q(X)\subseteq \Omega_0^q(X)$ the algebraic   direct sum 
 \[
 A_s^q(X) = \bigoplus_{\dim(C)=q} A_s(C),
 \]
   by $A^*_{s}(X)\subseteq \Omega^*_{0} (X)$ the direct sum of all $A^q _s (X)$, and by $A^*_{L^2,s}(X)\subseteq \Omega^*_{L^2} (X)$ the Hilbert space completion of $A^*_{s}(X)$.
\end{definition}

\begin{lemma}
\label{lem-invariance-of-A-s2}
For every $s\in (0,1]$ the operators 
\[
e_{sw}, e^{\diamond}_{sw}\colon \Omega^*(X)\longrightarrow \Omega^*(X)
\]
map $A_s^*(X) $ into itself.
\end{lemma}

 \begin{proof}
 Let $s \in (0,1]$. The operator $e_{sw}$ has the form
 \begin{equation}
 \label{eq-action-of-esw}
 e_{sw}\colon \beta \longmapsto \sum_{H\in \SAH(C)} s w(H) \, e^{s w(H)y_H}\, \alpha_H \wedge \beta 
 \end{equation}
 for $\beta \in \Omega^*(C)$; see Definition~\ref{def-e-differential}.  Recall that in this formula $\alpha_H \wedge \beta $  is to be viewed as a form on the cube $D_H$ that is cut by $H$ and contains $C$ as a codimension-one face; the form $\beta$ is pulled back to $D_H$ along the face projection $D_H {\to} C$.  If $C$ is cut by the hyperplanes $H_1,\dots, H_q$, then of course the  cube  $D_H$ is cut by the hyperplanes $H, H_1,\dots ,H_q$.  In addition, if $w_C$ is  viewed as a function on $D_H$ by pulling back along the face projection $D_H{\to} C$, then  
\[
w_{D_H} = w_C + w(H) x_H.
\]
  It follows that 
\[
sw_{D_H} = sw_C + sw(H) x_H = sw_C + s w(H) y_H + s w(H) ,
\]
and hence we find that 
 \begin{multline*}
 e_{sw}\colon 
e^{s w_C}dx_{H_1} \cdots dx_{H_q}  
\\
 \longmapsto \sum_{H\in \SAH(C)} (s w(H) e^{-sw(H)} ) e^{s w_{D_H}} dx_Hdx_{H_1}\cdots d x _{H_q} .  
 \end{multline*}
This proves the lemma for $e_{sw}$.  A similar computation shows that if $D$ is a $(q{+}1)$-cube cut by hyperplanes $H_0,\dots, H_q$, and if $C_j$ denotes the $q$-dimensional face of $D$ that is separated from the basepoint by $H_j$, then 
\begin{multline*}
e^\diamond_{sw} \colon e^{w_D}d_{H_0}\cdots dx_{H_q}
\\
\longmapsto 
\sum_{j=0}^q (-1)^j  \| e^{s w(H_j)x}\|_{L^2 (0,1)}^2 \cdot s w(H_j) e^{-s w(H_j)}
\\
\cdot e^{w_{C_j}}dx_{H_0}\cdots \widehat{dx_{H_j}}\cdots dx_{H_q}. 
\end{multline*}
This proves the lemma for $e_{sw}^{\diamond}$.
\end{proof}

Now $A_s^*(X)$ has an orthogonal complement in the inner product space
$\Omega^*_0(X)$, since each one-dimensional space 
$A_s (C)$ is certainly complemented in $\Omega_0^*(C)$, and it follows from Lemmas~\ref{lem-invariance-of-A-s1} and \ref{lem-invariance-of-A-s2} that the operator $s^{-1} D_{\mathrm{dR},s} $ preserves the direct sum decomposition of $ \Omega^*_0(X)$ into $A_s^*(X)$ and its orthogonal complement in $ \Omega^*_0(X)$.  On the former, we have,  
 in view of Lemma~\ref{lem-invariance-of-A-s1},
\[
 s^{-1} D_{\mathrm{dR},s}  \, = \, s^{-1} e_{sw} {+} s^{-1} e_{sw}^\diamond \, 
 \colon 
 A_s^* (X) \longrightarrow A_s^* (X) .
 \]
 Moreover it follows from \eqref{eq-action-of-esw} that 
 \begin{equation}
 \label{eq-eq-action-of-esw-again}
 s^{-1} e_{sw}\colon \beta \longmapsto \sum_{H\in \SAH(C)}   w(H) \, e^{s w(H)y_H}\, \alpha_H \wedge \beta 
 \end{equation}
 for $\beta \in A_s (C)$.  But when  $s{=}0$ the space $A_s^*(X)$   is the space $A^*(X)$ from Definition~\ref{def-A-space-of-X} that we used in our construction of the Julg-Valette cycle, and it is at least informally clear from \eqref{eq-eq-action-of-esw-again}  that as $s{\to}0$ the operator $ s^{-1} e_{sw}$ ``converges'' to the Julg-Valette differential. It follows  that,  in the same informal sense,  $s^{-1}D_{\mathrm{dR},s}$ converges to the Julg-Valette operator $D_{\mathrm{JV}}$. 
 
  So roughly speaking, to prove Theorem~\ref{thm-de-Rham-is-JV}  we need to:
\begin{enumerate}[\rm (i)]

\item Make the above  idea of convergence precise.

\item Show that the summand of operator $s^{-1} D_{\mathrm{dR},s}$ that acts  on the orthogonal complement $A_s^*(X)^\perp \subseteq \Omega^*_0(X)$ contributes nothing in $KK$-theory.
  
  \end{enumerate}
This is what we shall do rigorously below. An additional complication
is that the direct sum decomposition above is not $G$-equivariant, but
this too will be properly  taken into account in our proof of the
theorem. 
  
Let us now return to the issue of the domain for the operator $\D$ in
\eqref{eq-def-curly-D}.    First we define 
\[
\mathcal{H}_0 (C)=
\left\{\,
  \beta \colon [0,1] \to \Omega^{\ast}_0(C) \, \middle\vert \,
        \begin{tabular}{l}$\beta$ is continuous for the
            $C^\infty$-topology \\on $\Omega^*_0(C)$, and $\beta_s \in
            \A_s(C)$ for all \\sufficiently small $s$\end{tabular} 
\, \right\} ,
\]
and then we form the  algebraic direct sum 
\[
\mathcal{H}_{0} (X)= \bigoplus _{C} \mathcal{H}_0(C),
\]
which is a dense $C[0,1]$-submodule of $\mathcal{H}(X)$.  This will be
our domain. The additional smoothness and continuity hypotheses ensure
that $\D$ maps $\mathcal{H}_0 (X)$ into $\mathcal{H}(X)$.

\begin{theorem}
\label{thm-curly-d-regular}
The operator $\D$, with the above domain $\mathcal{H}_{0} (X)$,  is a regular and essentially self-adjoint operator on the Hilbert module $\mathcal{H}(X)$. 
\end{theorem}

To prove the theorem we shall use the following simple technical
result, which is proved by explicitly constructing the resolvent
operators for the operator in question.  For further information about
regularity for Hilbert module  operators, which is needed to guarantee
a reasonable spectral theory for these operators, see \cite{Lance}. 
\begin{lemma}
\label{lem-tech-regularity-result}
Let $\{ H_s\}_{s\in [0,1]}$ be a continuous field of Hilbert spaces
over $[0,1]$ with $H_0 {=}0$, and let $\mathcal E$ be a dense
submodule of its Hilbert $C[0,1]$-module of continuous sections. Let
$\{ D _s \} _{s \in [0,1]}$ be a family of   symmetric operators on
the   fibers of the field.  Assume that:
\begin{enumerate}[\rm (i)]

\item  If  $e\in \mathcal E$, then $e_s \in \operatorname{dom}(D_s)$  for every $s\in [0,1]$.

\item    If  $e\in \mathcal E$, then 
$ s\mapsto  D_s e_s$ is a continuous section of  $\{ H_s\}_{s\in [0,1]}$.  
 
\item For every $\delta >0$ the family $\{ D_s\}_{s\in [\delta, 1]}$ defines 
 an essentially self-adjoint and regular operator on the Hilbert $C[\delta,1]$-module of continuous sections of the restricted field  $\{ H_s\}_{s\in [\delta,1]}$, with domain the module  of all restrictions of sections in $\mathcal{E}$ to $[\delta, 1]$.
 
\end{enumerate}
Then   $\{ D_s\}_{s\in [0, 1]}$ defines an essentially self-adjoint and regular operator on the Hilbert $C[0,1]$-module of all continuous sections of $\{ H_s\}_{s\in [0,1]}$, with domain $\mathcal{E}$. \qed
\end{lemma}

\begin{proof}[Proof of Theorem~\ref{thm-curly-d-regular}]
If  $Q$ is a vertex in  $X$, then define 
\[
 \mathcal{H}(X)_Q =  \bigl \{\,
\beta  \in \mathcal{H} (X)\, \bigm | \, \text{$\beta_s  \in \Omega^{\ast}_{L^2}(X)_Q \quad \forall s\in [0, 1]$} 
\, \bigr  \} , 
\]
and also define 
\[
 \mathcal{H}_{0}(X)_{Q} = \mathcal{H}_{0}(X)  \cap \mathcal{H}(X)_Q.
\]
Then
\[
\mathcal{H}  (X) = \bigoplus_Q \mathcal{H}  (X) _Q 
\quad \text{and} \quad 
\mathcal{H}_{0} (X) = \bigoplus_Q \mathcal{H}_{0} (X) _Q  ,
\]
where the former is a direct sum decomposition of Hilbert modules, and
the latter is an algebraic direct sum according to which the operator
$\D $ decomposes as a direct sum, $\D = \oplus_Q \D _Q $. 
To show that  $ \D$ is a  regular operator, we only need to show the
same for each $\D _Q$.

Define a self-adjoint projection operator 
\[
\begin{gathered}
\mathcal{P}_Q\colon \mathcal{H}(X)_Q \longrightarrow \mathcal{H}(X)_Q
\\
\bigl ( \mathcal{P}\beta\bigr ) _s =  P_s \beta _s  ,
\end{gathered}
\]
 where $P_s$ is the orthogonal projection from $\Omega^*_{L^2}(X)_Q$ onto the finite-dimen\-sional   subspace 
 $ {A}^*(X)_Q= {A}^*(X)\cap \Omega^*_{L^2}(X)_Q $. 
 The operator $\mathcal{P}$  maps the domain $\mathcal{H}_0(X)_Q$ into itself, and thanks to Lemmas~\ref{lem-invariance-of-A-s1} and \ref{lem-invariance-of-A-s2} it  commutes with   $\D$. 

Let us  denote by
\[
\mathcal{A}(X)_Q \subseteq \mathcal{H}(X)_Q.
\] 
the range of the operator $\mathcal{P}_Q$.  It is an orthocomplemented Hilbert submodule, and indeed a finitely generated and projective module in its own right.
The restriction of $\D$ to $\mathcal{A}(X)_Q$ is the bounded and norm-continuous family  of self-adjoint operators  
\[
\begin{cases}
s^{-1} e_{sw} + s^{-1}e_{sw}^\diamond  & s > 0 \\
d_{\mathrm{JV}} + d^\diamond _{\mathrm{JV}} & s =0  ,
\end{cases}
\]
 and it is a regular and self-adjoint Hilbert module operator.

As for   the restriction of  $\D $   to the  orthogonal complement  $\mathcal{A}(X)_Q^\perp \subseteq \mathcal{H}(X)_Q$, whose domain is the orthogonal complement of the submodule $\mathcal{A}(X)_Q$ in $\mathcal{H}_0(X)_Q$, it follows immediately from the preceeding lemma that this too is a regular operator.
\end{proof}

In order to construct a $KK$-cycle we need a bit more than regularity: we also need an operator  with compact resolvent.  For this, we shall use the following equally easy but slightly stronger version of Lemma~\ref{lem-tech-regularity-result}: 

\begin{lemma}
\label{lem-tech-compact-resolvent-result}
Under the hypotheses of Lemma~\textup{\ref{lem-tech-regularity-result}}, assume in addition that:
\begin{enumerate}[\rm (i)]

\item For every $\delta >0$ the family $\{ D_s\}_{s\in [\delta, 1]}$ defines 
 an essentially self-adjoint and regular operator with compact resolvent.

\item  For every $K> 0$ there exists $\delta> 0$ so that if $s\in [0,\delta] $, then $D_s$ is bounded below by $K$ \textup{(}that is, $\|D_s e_s\| \ge K \| e_s\|$ for every $e_s \in \operatorname{dom}(D_s)$\textup{)}.
\end{enumerate}
Then   $\{ D_s\}_{s\in [0, 1]}$ defines an essentially self-adjoint and regular operator with compact resolvent on the Hilbert $C[0,1]$-module of all continuous sections of $\{ H_s\}_{s\in [0,1]}$, with domain $\mathcal{E}$. \qed
\end{lemma}

With this, we are finally ready to prove Theorem~\ref{thm-de-Rham-is-JV}. 

\begin{proof}[Proof of Theorem~\ref{thm-de-Rham-is-JV}]
We need  to show that the operator $\D$ has compact resolvent and is almost-equi\-variant, so that it determines a class in the equivariant group $KK_G(\C,C[0,1])$.

It follows from Lemma \ref{lemma_op}  that for $s\in (0,1]$ the operator $D_{\mathrm{dR},s} $  on $\Omega^{\ast}_{L^2}(X)_Q$ has square bounded below by 
\[
\begin{aligned}
\sum_{H\in \SAH(Q)}  w(H)^2 \frac{(1 - e^{-2sw(H)})}{2sw(H)} 
	& \geq \sum_{H\in \SAH(Q)}  w(H)^2 \frac{(1 - e^{-2w(H)})}{2w(H)}  \\
	& = \sum_{H\in \SAH(Q)}  \tfrac12 w(H)(1 - e^{-2w(H)}) ,
\end{aligned}
\]
and it follows from this  and the lemma above that   $\D $ has compact resolvent on $\mathcal{H}(X)$.

From the proof of Lemma \ref{lem-G-bounded-property-for-de-Rham}, we see that the operator family 
\[
\bigl \{ \, g\bigl (D_{\mathrm{dR},s} \bigr ) - D_{\mathrm{dR},s} \, \bigr\}_{s\in (0,1]}
\]
extends to a bounded operator on $\mathcal{H}(X)$, for every $g\in G$,  and defines a strongly continuous, locally bounded, bounded operator-valued function of $g\in G$, as required.
\end{proof}

Let us turn now to  Theorem~\ref{thm-de-Rham-has-gamma}. 
 We just showed   that $[D_{\mathrm{dR}}]{=}[D_{\mathrm{JV}}]$, and Theorem \ref{thm-JV=1} asserts that    $[D_{\mathrm{JV}}]{=}1_G$.  Therefore $[D_{\mathrm{dR}}]{=}1_G$, and so $D_{\mathrm{dR}}$ satisfies the condition (i)  in Definition~\ref{def_gamma}.  It remains to show that $D_{\mathrm{dR}}$ satisfies the condition (ii)  in Definition~\ref{def_gamma}.

\begin{proof}[Proof of Theorem~\ref{thm-de-Rham-has-gamma}]
We shall apply Theorem \ref{thm_gamma}. 
For the space $E$ we shall take the cubical complex $X$, of course. This is a cocompact model for the universal proper $G$-space. We shall use the natural non-degenerate representation of the $G$-$C^*$-algebra $C_0(X)$ on $\Omega^{\ast}_{L^2}(X)$ by pointwise multiplication on forms. 

For the algebra  $B$ we shall take the dense $G$-subalgebra of $C_0(X)$ consisting of compactly supported complex-valued functions that are smooth on each cube. It contains a cut-off function of $X$. For each $b\in B$, we have
\[
\begin{aligned}
\bigl [D_{\mathrm{dR}}, g(b)\bigr ]
	 & = \bigl [d_{ w} + e_{w} + d_{w}^\diamond + e_{ w}^\diamond, g(b)\bigr ]  \\
	& = \bigl [d_{ w} + d_{w}^\diamond, g(b)\bigr ]  + \bigl [e_{w} + e_{ w}^\diamond, g(b)\bigr ]  \\
	& = g(c(b))  + \bigl [e_{w} + e_{ w}^\diamond, g(b)\bigr ] ,
\end{aligned}
\]
where $c(b)$ denotes Clifford multiplication by the gradient of $b$ in each cube, which is a bounded operator,  uniformly bounded in $g$, and supported in the $g$-translate of the support of $b$. The operators $e_{w}, e_{w}^\diamond$ and $g(b)$ respect the decomposition \eqref{eq-normal-cube-decomp} of the Hilbert space $\Omega^{\ast}_{L^2}(X)$.  Moreover $b$ is non-zero on only finitely many blocks $\Omega^{\ast}_{L^2}(X)_Q$, and on these blocks $e_w$ and $e_{w}^\diamond$ are bounded. From this, we see that the commutator $[D,g(b)]$ extends to a bounded operator on $\mathcal{H}$ whose support is contained in $gK$ where $K$ is the union of cubes which intersect the support of $b$. 

In order to apply Theorem~\ref{thm_gamma}, it remains to show that the commutator 
\[
[e_{w}, g(b)] 
\]
is uniformly bounded in $g$. For this, choose $C$  so  that
\[
\bigl |b(x) - b(y)\bigr | \leq C\cdot \operatorname{distance}(x, y)
\qquad \forall x,y\in X.
\]
From the definition of $e_w$, we see that for all $\beta\in \Omega^\ast(C)$, 
\[
\begin{aligned}
\| [e_{w}, g(b)] \beta  \|^2 
	& \le \sum_{H\in \SAH(C)} C^2 \cdot \bigl  \|xw(H)e^{-w(H)x}\bigr  \|_{L^2(0,1)}^2   \cdot  \| \beta \|^2 \\
	&\le \dim(X) \cdot C^2  \cdot \sup\, \{\,  xe^{-x} :  x\geq0 \, \}  \cdot \| \beta \|^2 \\
	& \le \dim(X) \cdot C^2 \cdot \| \beta \|^2.
\end{aligned}
\]
So $[e_{w}, g(b)]$ is indeed uniformly bounded in $g\in G$, as required. 
\end{proof}

To conclude the paper, let us repeat the overall argument as it was
presented in Section~1.  Suppose $G$ is a second countable, locally
compact group acting on a bounded geometry $CAT(0)$-cubical space $X$
by automorphisms.  Theorems~\ref{thm-de-Rham-is-JV} and
\ref{thm-de-Rham-has-gamma}, combined with the Theorem~\ref{thm-JV=1},
imply that the identity element of Kasparov's representation ring
$R(G)$ is represented by a cycle with Property $(\gamma)$.  It follows
therefore from Theorem~\ref{thm_BCC} that the Baum--Connes conjecture
with coefficients holds for $G$.

\bibliography{Refs}
\bibliographystyle{alpha}

\end{document}